\documentclass{article}
\usepackage[preprint]{log_2026}			% for preprint version

\usepackage{booktabs}						% professional-quality tables
\usepackage{multirow}						% tabular cells spanning multiple rows
\usepackage{amsfonts}						% blackboard math symbols
\usepackage{graphicx}						% figures
\usepackage{duckuments}						% sample images

\usepackage[numbers,compress,sort]{natbib}	% for numerical citations

\usepackage{amssymb}
\usepackage{amsthm}
\newtheorem{theorem}{Theorem}
\newtheorem{lemma}[theorem]{Lemma}
\newtheorem{proposition}[theorem]{Proposition}
\newtheorem{corollary}[theorem]{Corollary}
\newtheorem{definition}[theorem]{Definition}

\title[Magnitude homology and Euler characteristics of DAGs]{Magnitude homology and Euler characteristics of directed acyclic graphs}

\author[Huntsman]{%
Steve Huntsman\\
\email{steve.huntsman@cynnovative.com}
}

\begin{document}

\maketitle

\begin{abstract}
We develop a scalable approach to computing the magnitude homology Euler characteristic for directed acyclic graphs based on decategorification. Along with motivating mathematical results and some simple controlled examples, we deploy the Euler characteristic in a proof of concept application to the dynamic analysis of multilayer perceptrons, recovering class-discriminative structure while holding simpler subgraph properties fixed. 
\end{abstract}

\section{Introduction}\label{sec:Introduction}

Magnitude is an effective notion of size for enriched categories, particularly generalized metric spaces \citep{leinster2013magnitude,leinster2017magnitude,leinster2021entropy}, that is finding applications to data science and machine learning \citep{bunch2021weighting,adamer2024magnitude,limbeck2024metric,limbeck2026geometry}, optimization \citep{pereverdieva2025comparative,emmerich2026magnitude}, and even chemistry \citep{bi2024persistent,bi2025topological}. The theory of magnitude homology is a \emph{categorification} \citep{hepworth2017categorifying,leinster2021magnitude} that is to magnitude roughly as singular homology is to ordinary Euler characteristic: the process of \emph{decategorification} recovers Euler characteristics from homologies. 

Although magnitude homology has been proposed for analyzing neural networks and connectomes for several years, to our knowledge, the only published analysis in this vein is \citet{kachura2025blurred}. In this paper, we develop a scalable approach to computing the magnitude homology Euler characteristic for directed acyclic graphs (DAGs) based on decategorification and a proof of concept application to the dynamic analysis of multilayer perceptrons (MLPs).

The paper is organized as follows. \S \ref{sec:preliminaries} covers preliminaries; \S \ref{sec:magnitudeHomology} introduces magnitude homology; \S \ref{sec:calculations} calculates magnitude homology for some DAGs; \S \ref{sec:computing} introduces an efficient algorithm for the Euler characteristic; and \S \ref{sec:layered} discusses layered DAGs, from some theoretical results to basic examples to the dynamic analysis of MLPs. After a closing remark about a plausible application to point clouds in \S \ref{sec:remarks}, the appendix \S \ref{sec:code} gives MATLAB code.

\section{Preliminaries}\label{sec:preliminaries}

Our treatment introduces magnitude homology below and gestures at a reasonable amount of self-containment, but we practically assume generic familiarity with homology and Euler characteristic at the levels of \cite{hatcher2002algebraic} and \cite{robinson2026computational}. For an abstract simplicial complex (i.e., a family of subsets of an ambient set that is closed under inclusion), recall that the Betti number $\beta_k$ is the rank of the $k$th homology group: the corresponding Euler characteristic is $\chi := \sum_k (-1)^k \beta_k$.

For simplicity, we assume from the outset that digraphs, posets, etc. are all finite. Recall that transitively reduced DAGs correspond bijectively to (Hasse diagrams of) finite posets. DAGs generally carry more information than posets. A DAG $D$, its transitive reduction (i.e., the minimal DAG with the same reachability relation), and its transitive closure (i.e., the maximal DAG with the same reachability relation) all correspond to the same poset, which we also abusively denote by $D$ despite a non-bijective correspondence: we also deploy partial order relations on DAGs without comment. Recall that a \emph{st-DAG} (also known as a \emph{two-terminal DAG}) is a DAG with a unique \emph{source} of indegree zero and a unique \emph{target} of outdegree zero. Any DAG can be made into a st-DAG by adjoining a source and target, just as any poset can be made an interval in a larger poset that adjoins a bottom $\hat 0$ and top $\hat 1$. 

Any directed graph has the shortest directed path quasimetric, which we write as $d$. 

The \emph{order complex} $\Delta(D)$ of a poset $D$ is the abstract simplicial complex with $(k-1)$-simplices given by $k$-element chains in $D$ \citep{stanley1997enumerative,wachs2007poset}. The \emph{M{\"o}bius function} $\mu_D$ has $\mu_D(u,u) = 1$ and $\mu_D(u,v) = -\sum_{u \le w < v} \mu_D(u,w)$ for $u < v$. The \emph{Philip Hall theorem} is $\mu_{\hat D} = \tilde \chi(\Delta(D))$, where $\hat D$ is the result of adjoining a (new) bottom and top to $D$, and where $\tilde \chi$ is the reduced simplicial Euler characteristic, with $\tilde \chi(\varnothing) = -1$ by convention, and $\tilde{\chi}(\Delta(D)) = \chi(\Delta(D)) - 1$.

It turns out that magnitude homology of a poset is the (simplicial) homology of its order complex \citep{kaneta2021magnitude,roff2025iterated}. Thus we immediately have that the magnitude homology Euler characteristic of a poset (though not of a DAG in general) is given by a M{\"o}bius function, which admits a very efficient computation \citep{pegolotti2022fast}. However, the decategorification of magnitude homology \eqref{eq:homologyDecategorification} suggests alternative algorithms for computing the Euler characteristic for DAGs and st-DAGs that we shall develop and apply below.

\section{Background on magnitude homology}\label{sec:magnitudeHomology}

We follow the concise formulation of magnitude homology due to \citet{hepworth2018magnitude}. Let $(X,d)$ be a Lawvere metric space, i.e., $d$ is an extended quasipseudometric.
\footnote{Here ``extended'' means $d : X \times X \rightarrow [0,\infty]$; ``quasi'' means $d$ need not be symmetric, and ``pseudo'' means that $d(x,x)$ need not be zero. A Lawvere metric space is best thought of here as a category enriched over the poset $([0,\infty],\ge)$, not least because the theory of magnitude (co)homology applies to enriched categories with a suitable notion of ``size.''}
A \emph{$k$-simplex} in $X$ is an ordered tuple $x^{(k)} := (x_0,\dots,x_k) \in X^{k+1}$ such that adjacent entries are distinct; its \emph{length} is $\lambda(x^{(k)}) := \sum_{j=1}^k d(x_{j-1},x_j)$. The $\mathbb{R}$-graded \emph{magnitude chain complex} over the coefficient ring $R$ has $k$-chains 
\[MC_{k,L}(X) := R \left \{ x^{(k)} \text{ a $k$-simplex in $X$} : \lambda(x^{(k)}) = L \right \}\] 
and differential $\partial_k : MC_{k,L}(X) \rightarrow MC_{k-1,L}(X)$ given by 
$\partial_k := \sum_{j=1}^{k-1} (-1)^j \partial^{(j)}$,
with 
\[\partial^{(j)}(x^{(k)}) := \begin{cases} \nabla_j x^{(k)} & \text{ if } d(x_{j-1},x_{j+1}) = d(x_{j-1},x_j) + d(x_j,x_{j+1}) \\ 0 & \text{ otherwise,} \end{cases}\] 
and where $\nabla_j$ acts on a tuple by deleting the corresponding entry. The appropriate notion of chain map is induced by distance-nonincreasing maps.

\begin{definition}\label{def:MH}
The \emph{magnitude homology} $MH(X)$ is the homology of $MC(X)$. 
\end{definition}
Conveniently for a digraph $D$, there is a direct sum decomposition of the form 
\begin{equation}\label{eq:directSum}
MC_{\bullet,L}(D) = \bigoplus_{s,t \in V(D)} MC_{\bullet,L}^{(s,t)}(D)
\end{equation}
where the direct summands on the right are subcomplexes generated by simplices with initial and terminal entries $s$ and $t$, respectively \citep{asao2020geometric}. This provides a very specific algebraic structure and a reduction in computational complexity that makes magnitude homology particularly convenient for unweighted (or $\mathbb{Z}_+$-weighted) digraphs.
\footnote{
Notably, magnitude homology is related to Vietoris-Rips homology via ``blurring'' that essentially replaces equalities of the form $\lambda = L$ in the definition of the chain complex with inequalities $\lambda \le L$ \citep{otter2018magnitude,cho2019quantales,govc2021persistent}, and it is also closely related to path homology \citep{asao2022magnitude}. 
}

In particular, magnitude homology is useful for characterizing DAGs, as we shall demonstrate. In light of the direct sum decomposition above, st-DAGs are particularly good candidates for characterizations using magnitude homology.

\section{Calculations of magnitude homology for some DAGs}\label{sec:calculations}

In this section, we detail some calculations that motivate applications by illustrating the impact of bottlenecks in st-DAGs and the homology of graphs that underlie MLPs. 

% Claude Opus 4.5 was able to give a proof of Theorem \ref{thm:pinchHomology} after we proposed it on the basis of numerical experiments: the version below is our adaptation.
% % <TODO>
% % \footnote{See \url{https://claude.ai/share/b2237ccf-f8c0-4d8d-9986-72172d9829a3}. Note that this begins with a much more abstract purported proof that we have not evaluated, and which may or may not be correct.}
% % </TODO>

\begin{theorem}\label{thm:pinchHomology}
Let $D$ be a finite st-DAG with source $s$ and target $t$, let $D'$ be a finite st-DAG with source $s'$ and target $t'$, and let $P$ be the st-DAG obtained by identifying $t$ and $s'$. Then for all $k$ and $L$, 
\begin{equation}\label{eq:pinchHomology}
MH_{k,L}^{(s,t')}(P) = 0.
\end{equation}
\end{theorem}

\begin{proof}
By Theorem 2.10 of \citet{hatcher2002algebraic}, a chain homotopy between zero and the identity implies trivial homology. It therefore suffices to produce a chain homotopy $h_{k,L}: MC_{k,L}^{(s,t)}(P) \rightarrow MC_{k+1,L}^{(s,t)}(P)$ satisfying 
\begin{equation}\label{eq:chainHomotopyIdentityNull}
(\partial_{k+1,L} \circ h_{k,L} + h_{k-1,L} \circ \partial_{k,L})(x) \equiv x.
\end{equation}

Let $v$ be the vertex in $P$ that is identified with $t$ and $s'$ (we will somewhat abusively make this identification for convenience below), and let $x$ be a generating simplex of $MC_{k,L}^{(s,t)}(P)$. That is, $x = (s = x_0, \dots, x_k = t')$, with $\sum_{j=1}^k d(x_{j-1},x_j) = L$. Call $x_1, \dots, x_{k-1}$ the \emph{internal vertices} of $x^{(k)}$. Since the differential removes internal vertices that lie on geodesics and $v$ is always such a vertex when it is present, a map $h$ that inserts $v$ into simplices with the right sign is a good candidate to satisfy \eqref{eq:chainHomotopyIdentityNull}. 

Specifically, let $h(x) := 0$ if $v$ is an internal vertex of $x$; otherwise, set 
\[h(x) := (-1)^{i+1} \cdot (x_0, \dots, x_i, v, x_{i+1}, \dots, x_k)\]
where $i$ is the largest index s.t. $x_i \in D$, so that $x_{i+1} \in D' \setminus \{v\}$. Extend $h$ to $MC_{\bullet,\bullet}^{(s,t)}(P)$ by linearity. The proof that $\partial \circ h + h \circ \partial$ is the identity amounts to checking the cases where $v$ i) is and ii) is not an internal vertex of $x$.

Suppose first that $v$ is an internal vertex of $x$, say $v = x_\ell$. Then $h(x) = 0$, so $(\partial \circ h)(x) = 0$. Meanwhile, $\partial(x) = (-1)^\ell \cdot (x_0, \dots, x_{\ell-1}, x_{\ell+1}, \dots, x_k) + y$, where here $y$ is a linear combination of simplices that contain $v$ and that is therefore annihilated by $h$. It follows that
\begin{align}
(h \circ \partial)(x) & = (-1)^\ell \cdot h(x_0, \dots, x_{\ell-1}, x_{\ell+1}, \dots, x_k) \nonumber \\
& = (-1)^\ell \cdot (-1)^{\ell-1+1} \cdot h(x_0, \dots, x_{\ell-1}, v, x_{\ell+1}, \dots, x_k) \nonumber \\
& = x.\nonumber
\end{align}
Therefore \eqref{eq:chainHomotopyIdentityNull} holds if $v$ is an internal vertex of $x$.

Suppose now that $v$ is not an internal vertex of $x$, and as before let $i$ is the largest index s.t. $x_i \in D$, so that $x_{i+1} \in D' \setminus \{v\}$. Then 
\begin{align}
(\partial \circ h)(x) & = \sum_{j=1}^i (-1)^{i+1} \cdot (-1)^j \cdot \partial^{(j)} (x_0, \dots, x_i, v, x_{i+1}, \dots, x_k) \nonumber \\
& + (-1)^{i+1} \cdot (-1)^{i+1} \cdot \partial^{(i+1)} (x_0, \dots, x_i, v, x_{i+1}, \dots, x_k) \nonumber \\
& + \sum_{j=i+1}^{k-1} (-1)^{i+1} \cdot (-1)^{j+1} \cdot \partial^{(j+1)} (x_0, \dots, x_i, v, x_{i+1}, \dots, x_k) \nonumber
\end{align}
and
\begin{align}
(h \circ \partial)(x) & = \sum_{j=1}^i (-1)^i \cdot (-1)^j \cdot \partial^{(j)} (x_0, \dots, x_i, v, x_{i+1}, \dots, x_k) \nonumber \\
& + \sum_{j=i+1}^{k-1} (-1)^i \cdot (-1)^{j+1} \cdot \partial^{(j+1)} (x_0, \dots, x_i, v, x_{i+1}, \dots, x_k).  \nonumber
\end{align}
After cancellations, we have that 
\begin{align}
(\partial \circ h + h \circ \partial)(x) & = (-1)^{i+1} \cdot (-1)^{i+1} \cdot \partial^{(i+1)} (x_0, \dots, x_i, v, x_{i+1}, \dots, x_k) \nonumber \\
& = x, \nonumber
\end{align}
establishing \eqref{eq:chainHomotopyIdentityNull} when $v$ is not an internal vertex of $x$. The result follows.
\end{proof}

\begin{definition}\label{def:MLP}
Let $K^\rightarrow_{n_1,\dots,n_M}$ denote the DAG corresponding to the architecture of a MLP with $M$ layers of successive widths $n_1, \dots, n_M$.  
That is, $V(K^\rightarrow_{n_1,\dots,n_M}) = \cup_{m=1}^M K_m$ with $K_m = \{m\} \times [n_m]$, and 
\[A(K^\rightarrow_{n_1,\dots,n_M}) = \{(v,v') : v \in K_m, \, v' \in K_{m+1}, \, 1 \leq m \leq M-1\}.\] 
For $v \equiv (\{m\} \times j_m) \in K^\rightarrow_{n_1,\dots,n_M}$, let $\Lambda(v) := m$, i.e., $\Lambda$ indicates the layer of $K^\rightarrow_{n_1,\dots,n_M}$ that its argument belongs to.
\end{definition}

Write $\beta_{k,L}^{(s,t)}$ for the Betti numbers (i.e., the ranks) of $MH_{k,L}^{(s,t)}$.

\begin{lemma}\label{lem:mlpHomologyST}
\begin{equation}\label{eq:mlpHomologyST}
\beta_{k,L}^{(s,t)}(K^\rightarrow_{n_1,\dots,n_M}) = 
	\begin{cases}
		\prod_{m = \Lambda(s)+1}^{\Lambda(t)-1} (n_m-1) & \text{if } k = L \text{ and } d(s,t) = L; \\
		0 & \text{otherwise}.
	\end{cases}
\end{equation}
\end{lemma}

% Claude Opus 4.5 was able to give a proof of Lemma \ref{lem:mlpHomologyST} that explained our own numerical experiments: the version below is our adaptation.
% % <TODO>
% % \footnote{See \url{https://claude.ai/share/4855b193-6948-4ed2-a816-0eb7eb85553a}.}
% % </TODO>

\begin{proof}
It suffices to show that $\dim \ker \partial_{k,k}^{(s,t)}$ is given by the right hand side of \eqref{eq:mlpHomologyST}, since $MC_{k+1,k}^{(s,t)} = 0$. With this in mind, consider a chain of the form
\[\alpha \equiv (s) \otimes \bigotimes_{m=\Lambda(s)+1}^{\Lambda(t)-1} \left ( \sum_{i_m=1}^{n_m} \alpha^{(m)}_{i_m} \cdot (x_m) \right ) \otimes (t)\] 
where $\sum_{i_m=1}^{n_m} \alpha^{(m)}_{i_m} = 0$ and $\Lambda(x_m) = m$ for all $m$. For $k = \Lambda(t)-\Lambda(s) = L$, each term of $\partial_{k,L}^{(s,t)}$ applied to $\alpha$ gives zero, since $\sum_{i_m=1}^{n_m} \alpha^{(m)}_{i_m} = 0$ factors out of the $m$th term. That is, $\alpha \in \ker \partial_{k,L}^{(s,t)}$. The dimension of the space of such chains is given by \eqref{eq:mlpHomologyST}. Meanwhile, the same reasoning as above shows that any chain not in this space has nonzero boundary.
\end{proof}

Summing over sources and targets gives the following corollary:
\footnote{For comparison, the path homology Betti numbers are $\beta_k = \delta_{k,M-1} \cdot \prod_{m=1}^M (n_m-1)$ \citep{chowdhury2019path}.}
\begin{corollary}\label{cor:mlpHomology}
\begin{equation}\label{eq:mlpHomology}
\beta_{k,L}(K^\rightarrow_{n_1,\dots,n_M}) = 
	\begin{cases}
		\sum_{m=1}^M n_m & \text{if } k = L = 0; \\
		\sum_{m=1}^{M-L} n_m \cdot \left ( \prod_{\ell=m+1}^{m+L-1} (n_\ell-1) \right ) \cdot n_{m+L} & \text{if } k = L > 0; \\
		0 & \text{otherwise}.
	\end{cases}
\end{equation}
\end{corollary}
For example, 
$\beta_{0,0}(K^\rightarrow_{5,4,3,2}) = 5+4+3+2 = 14$;
$\beta_{1,1}(K^\rightarrow_{5,4,3,2}) = 5 \cdot 4 + 4 \cdot 3 + 3 \cdot 2 = 38$;
$\beta_{2,2}(K^\rightarrow_{5,4,3,2}) = 5 \cdot 3 \cdot 3 + 4 \cdot 2 \cdot 2 = 61$;
$\beta_{3,3}(K^\rightarrow_{5,4,3,2}) = 5 \cdot 3 \cdot 2 \cdot 2 = 60$;
and all other Betti numbers are zero. Similarly,
$\beta(K^\rightarrow_{2,11,3,7,5}) = \text{diag}(28,111,304,940,1200,0,0,\dots)$;
and $\beta(K^\rightarrow_{10,2,8,4,6}) = \text{diag}(30,92,280,532,1260,0,0,\dots)$.

We also get a simple formula for the source-target Euler characteristic when $n_1 = n_M = 1$, so that $K^\rightarrow_{1,n_2,\dots,n_{M-1},1}$ is a st-DAG:
\begin{corollary}\label{cor:mlpEuler}
\begin{equation}\label{eq:mlpEuler}
\chi_{\bullet,M-1}^{(s,t)}(K^\rightarrow_{1,n_2,\dots,n_{M-1},1}) = (-1)^{M-1} \prod_{m=2}^{M-1} (n_m-1).
\end{equation}
\end{corollary}

\section{Computing the Euler characteristic}\label{sec:computing}

For a st-DAG $D$ the general (de)categorification formula \citep{leinster2021magnitude}
\begin{align}
\label{eq:homologyDecategorification}
((\exp[-\tau d])^{-1})_{st} & \overset{\tau \gg 0}{=} \sum_{k,L} (-1)^k \beta_{k,L}^{(s,t)} \exp(-\tau L) \nonumber \\ 
& = \sum_L \chi_{\bullet,L}^{(s,t)} \exp(-\tau L)
\end{align}
is particularly convenient: it exactly produces the magnitude homology Euler characteristic via Laplace transform \emph{without ever computing magnitude homology explicitly}, and it converges for all $\tau > 0$, since only finitely many Betti numbers are nonzero.
\footnote{
The proof of Theorem 4.9 in \citet{govc2021persistent} details the link with blurred magnitude homology.
}
\footnote{
An easy hack to work over weighted graphs is to replace single weighted arcs with series of unweighted arcs. This works exactly for integer weights and approximately for positive weights. 
}

In a bit more detail, write 
\[A_{ij} := \exp(-\tau_i L_j),\] 
where $L_j$ are the unique lengths of paths from $s$ to $t$, and 
\[b_i := ((\exp[-\tau_i d])^{-1})_{st}.\] 
Then \eqref{eq:homologyDecategorification} is of the form $Ax = b$ where $x_j := \chi_{\bullet,L_j}^{(s,t)}$. In practice, the $L_j$ (and $d$) are integral and determined by $D$, but we can freely choose the $\tau_i$. In particular, if $A$ is invertible over $\mathbb{Q}$, then $b$ is rational. Thus taking $\tau_i := \log q_i$ for $q_i \in \mathbb{Q}$ gives $A_{ij} = q_i^{-L_j} \in \mathbb{Q}$ and we can compute the Euler characteristic accordingly without resorting to numerically truculent Laplace transform inversions \citep{abate2006unified,dingfelder2015improved} or Prony-type methods for structured function reconstruction \citep{plonka2018numerical}. 

Instead, we have a numerically truculent linear algebra problem from the outset. $A$ is a generalized Vandermonde matrix, inheriting all of the badly conditioned behavior of an ordinary Vandermonde matrix \citep{gautschi2011optimally} without the benefit of a stable solver like Algorithm 4.6.2 in \citet{golub2013matrix} that ordinarily provides compensation. As such, it is not possible to get reliable exact solutions with a straightforward linear solver, and it is not feasible to solve its $\mathbf{NP}$-hard integer least squares variant $\arg \min_{x \in \mathbb{Z}^{\text{dim}(x)}} \|Ax-b\|^2$ by reduction to a closest vector problem. There are instead two practical intermediate approaches: a very robust, complex, and relatively slow and still heuristic method of performing lattice reduction in arbitrary precision arithmetic, and a reasonably robust, simple, and very fast heuristic using the nearest plane algorithm that returns the closest lattice point to a target after Gram-Schmidt \citep{babai1986lovasz,chang2013effects}. For fast/batched floating-point calculations we use the latter approach, with 
\[q_i = (1+|\mathcal{L}|^{-1})^i \in \mathbb{Q},\]
where here $\mathcal{L}$ is the set of unique lengths of paths from $s$ to $t$. 

A variant of this algorithm that takes $\tau := -\log q$ (note the change of sign relative to the above discussion) so that $\exp(-\tau d(u,v)) = q^{d(u,v)}$ gives
\begin{equation}
\label{eq:homologyDecategorificationPoly}
(Z^{-1})_{st} = \sum_L \chi_{\bullet,L}^{(s,t)} q^L
\end{equation}
which can be solved exactly over $\mathbb{Z}[q]$ \emph{versus} $\mathbb{Q}$. Using polynomial arithmetic toward this end \emph{versus} floating-point arithmetic dispenses with the numerical truculence of Vandermonde matrices, obviates the need for (e.g.) the nearest plane heuristic, and avoids \emph{ad hoc} tuning for $q_i$, all of which can contribute to errors at scale. On the other hand, the floating-point approach is usually faster. Either approach is dramatically faster than a direct computation of magnitude homology. In batched operation, we typically use the floating-point approach most of the time and occasionally verify outputs using the polynomial arithmetic approach: this gives us the benefits of speed with reasonable assurance of correctness.

Experiments with both our implementation of the floating-point algorithm (\S \ref{sec:floatingPointArithmetic}) and an adaptation for polynomial arithmetic (\S \ref{sec:polynomialArithmetic}) consistently give results for DAGs that agree with the theoretical results above for smaller instances, though for large instances (roughly, on the order of a thousand or more vertices) they sometimes disagree due to numerics, in which case the polynomial arithmetic implementation is the one to trust: for applications such as \S \ref{sec:MLPEulerExamples}, the batched operation approach above determines this. Both implementations in \S \ref{sec:code} are under 100 lines, much of it comments, whitespace, warnings, and assertions.

\section{Layered DAGs}\label{sec:layered}

A finite st-DAG is \emph{layered} if for $u \le v$, every path from $u$ to $v$ has length $d(u,v)$. A layered DAG $D$ is graded by $\Lambda : V(D) \rightarrow \mathbb{N}$ with $\Lambda(v) - \Lambda(u) = d(u,v)$ for $u \le v$. 

\begin{theorem}\label{thm:layeredChain}
Let $D$ be a layered DAG and $u \le v$. Then
\begin{equation}\label{eq:layeredChain}
MC_{k,d(u,v)}^{(u,v)}(D) \cong \tilde C_{k-2}(\Delta((u,v)))
\end{equation}
where $\tilde C_k$ indicates the augmented simplicial chain complex of $\Delta((u,v))$ taken as an open interval order complex. Furthermore, $MC_{k,L}^{(u,v)}(D) = 0$ for $L \ne d(u,v)$.
\end{theorem}

% Claude Opus 4.7 produced the original version of this proof, which we have adapted.
% % <TODO>
% % \footnote{See \url{https://claude.ai/share/98f55def-f72a-4bfc-b0a3-46bc36ff8864}}
% % </TODO>

\begin{proof}
The vanishing of $MC_{k,L}^{(u,v)}(D)$ for $L \ne d(u,v)$ follows immediately from the observation that since $D$ is layered, every path from $u$ to $v$ has length $d(u,v)$. The same observation yields that generators of $MC_{k,d(u,v)}^{(u,v)}(D)$ are precisely the $k$-simplices of length $\Lambda(v) - \Lambda(u) = d(u,v)$. Since the endpoints $u$ and $v$ are given, generators of $MC_{k,d(u,v)}^{(u,v)}(D)$ are in bijective correspondence with $(k-2)$-simplices in $\Delta((u,v))$. For $k = 1$, this trivializes to the augmentation $\tilde C_{-1} \cong R$, where as usual $R$ is the coefficient ring.

Now consider the boundary operator acting on $x = (x_0, \dots , x_k)$ with $x_{j-1} < x_j < x_{j+1}$ for $1 \le j \le k-1$. We automatically have that $d(x_{j-1},x_{j+1}) = d(x_{j-1},x_j) + d(x_j,x_{j+1})$, so $\partial^{(j)} \equiv \nabla_j$ and $\partial_k$ corresponds to the simplicial differential (up to sign). The isomorphism \eqref{eq:layeredChain} follows. 
\end{proof}

The Philip Hall theorem yields

\begin{corollary}\label{cor:mobius}
Let $D$ be a layered DAG and $u \le v$. Then
\begin{equation}\label{eq:mobius}
\chi_{\bullet,d(u,v)}^{(u,v)}(D) = \tilde \chi (\Delta((u,v))) = \mu_D(u,v),
\end{equation}
and $\chi_{\bullet,L}^{(u,v)}(D) = 0$ for $L \ne d(u,v)$. Moreover, $MH_{k,d(u,v)}^{(u,v)}(D) \cong \tilde H_{k-2}(\Delta((u,v)))$.
\end{corollary}

\subsection{Multilayer perceptron DAGs}\label{sec:mlp}

The MLP DAG $K^\rightarrow_{n_1,\dots,n_M}$ is layered, with grading $\Lambda$ indicating the layer in an obvious way. To fix a st-DAG, it is convenient to take $n_1 = 1 = n_M$. Inspection reveals that
$\Delta(K^\rightarrow_{1,n_2,\dots,n_{M-1},1}) \cong [n_2] * \dots * [n_{M-1}]$, where an order complex and the usual simplicial join are respectively indicated on the left and right sides. Deploying the appropriate K\"unneth theorem (see, e.g., Theorem 5.1.4 of \citet{wachs2007poset}) with the observation $\tilde \beta_k([n]) = \delta_{k0} \cdot (n-1)$ yields the following more elegant accounting of Lemma \ref{lem:mlpHomologyST}.
\footnote{This is substantially simpler than the analogous situation in path homology \citep{chowdhury2019path}.}

\begin{proposition}\label{pro:mlpHomologyST}
\begin{equation}\label{eq:Kunneth}
\tilde \beta_k([n_2] * \dots * [n_{M-1}]) = \delta_{k,M-3} \cdot \prod_{m=2}^{M-1} (n_m-1).
\end{equation}
\end{proposition}

\subsection{Basic examples}\label{sec:BasicEulerExamples}

Figures \ref{fig:euler_structured} and \ref{fig:euler_N_1000000_n_6_L_3_E_10} show that $\chi$ increases with parallel structure in layered st-DAGs, even when  structural features such as the number of vertices and arcs per layer are held fixed.

\begin{figure}[htbp]
    \centering
    \includegraphics[width=\textwidth, trim={0 125 0 0mm}, clip]{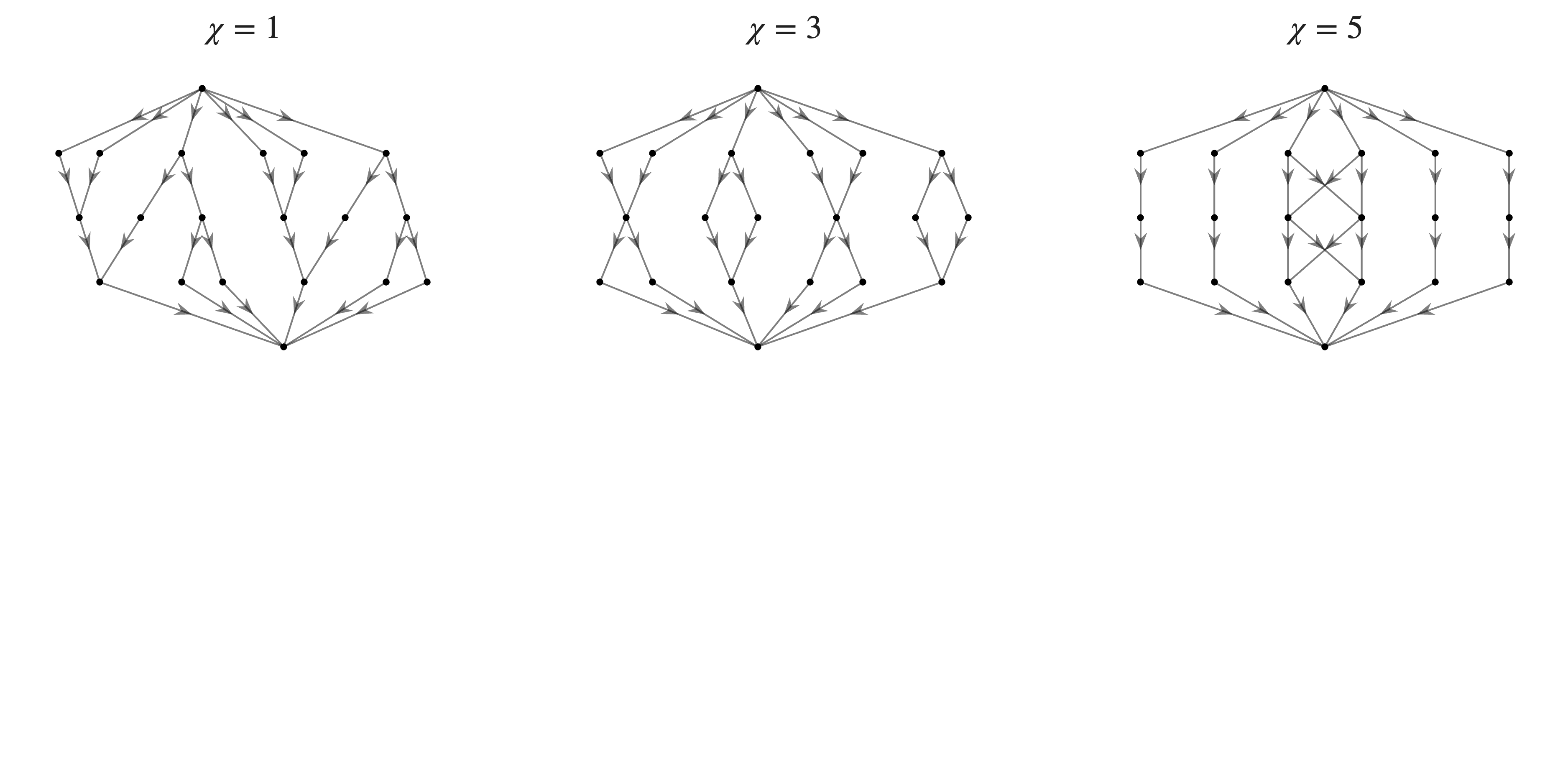}
        \caption{$\chi \equiv \chi_{\bullet,d(s,t)}^{(s,t)}$ for layered st-DAGs with six vertices in each of three intermediate layers and eight arcs between intermediate layers.}
    \label{fig:euler_structured}
\end{figure}

\begin{figure}[htbp]
    \centering
    \includegraphics[width=\textwidth, trim={0 5 0 0mm}, clip]{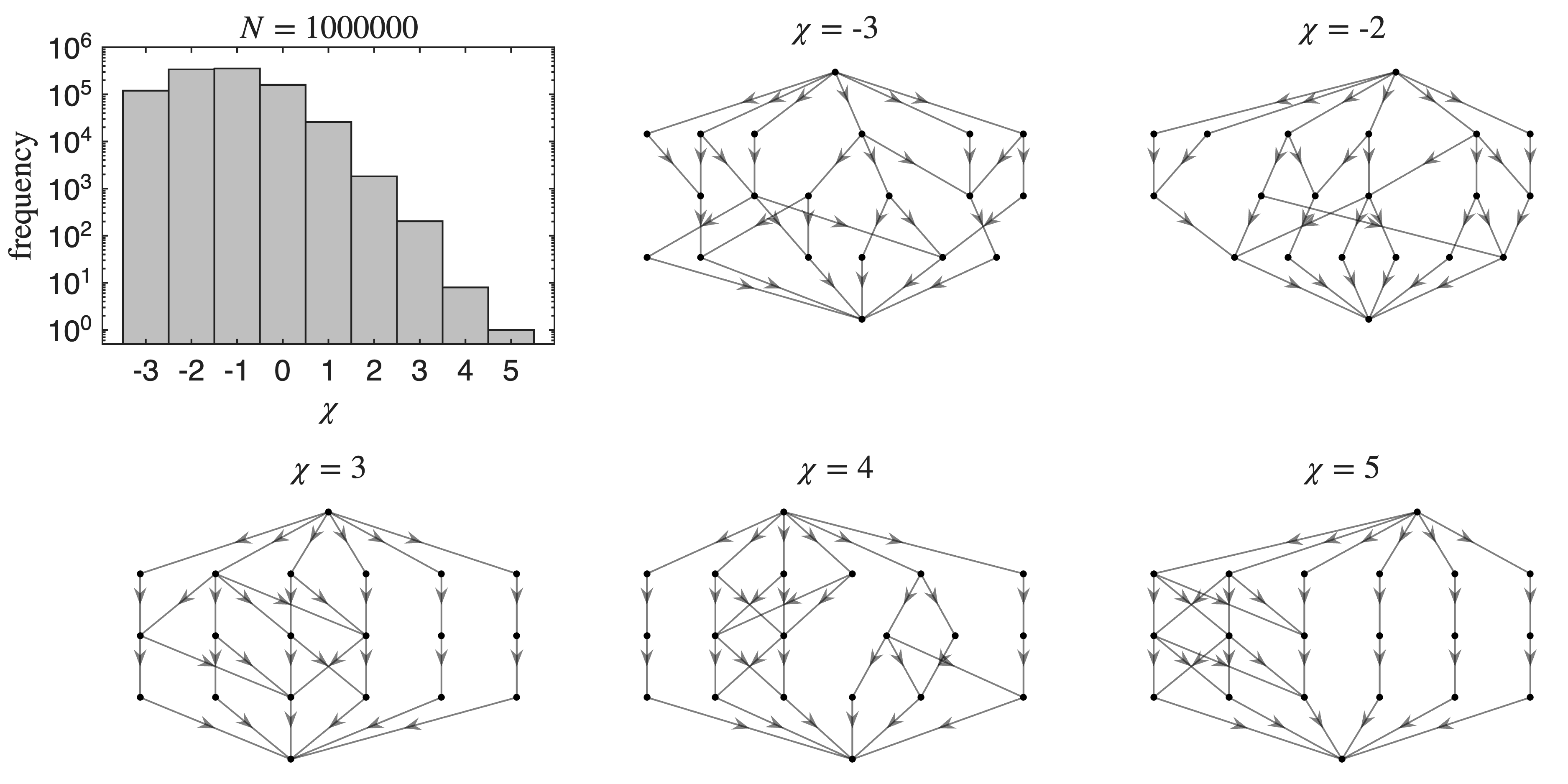}
        \caption{Upper left: a histogram of $\chi \equiv \chi_{\bullet,d(s,t)}^{(s,t)}$ for randomly sampled layered st-DAGs with six vertices in each of three intermediate layers, and ten arcs between adjacent intermediate layers. Other panels: representative st-DAGs for indicated $\chi$.}
    \label{fig:euler_N_1000000_n_6_L_3_E_10}
\end{figure}

\subsection{Dynamic analysis of MLPs}\label{sec:MLPEulerExamples}

Topological methods to analyze neural networks have recently proliferated \citep{ballester2026topological}, but the usual tools of persistent homology are not suited for structural analyses of neural networks \emph{qua} digraphs. Meanwhile, path homology \citep{chowdhury2019path} scales badly and lacks the granularity of magnitude homology's direct sum decomposition \eqref{eq:directSum}.

Following the thread of \S \ref{sec:BasicEulerExamples}, we construct and compute Euler characteristics for a family of layered st-DAGs associated with a MLP and input data. These st-DAGs have fixed numbers of vertices and fixed numbers of arcs per layer. This ensures that any features brought out by Euler characteristics are not attributable to any coarse graph structure.
\footnote{In previous investigations not detailed here, we computed Euler characteristics of st-DAGs that were not controlled in the way described above. We found that entries of the MLP's confusion matrix were very strongly correlated with the (logarithms of) average Euler characteristics computed over data corresponding to the confusion matrix entries. This was simply because the st-DAGs aggregated data over inputs and so grew according to the size of the input data.}

The basic st-DAG construction is simple and flexible. For a given source/input neuron $s$, target/output neuron $t$, and input data $X$, we consider a \emph{node score} $\nu_{s,t,X} : V(K^\rightarrow) \rightarrow \mathbb{R}$ and an \emph{arc score} $\alpha_{s,t,X} : A(K^\rightarrow) \rightarrow \mathbb{R}$. For convenience, we restrict consideration to MLPs with hidden layers of a fixed width and set a fixed number $k$ of neurons per layer to retain in each st-DAG. For each layer, the neurons with the $k$ highest values of $\nu$ are retained. The input and output neurons are connected to all $k$ retained neurons in the first and last hidden layers, respectively. Retained neurons are connected across adjacent hidden layers by $2k$ arcs. These arcs include: i) an arc from a retained neuron $u$ in layer $\ell$ to the neuron $u'$ in layer $\ell+1$ that produces the highest value of $\alpha$; ii) an arc to a retained neuron $v'$ in layer $\ell+1$ from the neuron $v$ in layer $\ell$ that produces the highest value of $\alpha$; and iii) any surplus arcs required to spend the budget of $2k$ arcs between adjacent hidden layers are chosen by the highest values of $\alpha$ that are otherwise unaccounted for. 

For a ReLU MLP with $M$ layers (including input and output layers), write $W^{(\ell)}$ for the weight (sub)matrix connecting layers $\ell$ and $\ell+1$, and write $a_u(x)$ for the activation of neuron $u$ on input $x$. If $u$ is an input neuron, then $a_u(x) = x_u$ is the appropriate component of $x$. Let $C_t$ be the subset of inputs in $X$ that are classified as $t$.

Below, we consider the specific node score
\[\nu_{s,t,X}(u) = \frac{1}{|C_t|} \sum_{x \in C_t} 1_{\text{supp} (a_{(u)})}(x) \cdot | a_s(x) |, \]
and the specific arc score 
\[\alpha_{s,t,X}(u,u') = 
\begin{cases} 
| W^{(1)}_{uu'} | \cdot \nu_{s,t,X}(u') & \text{if } s = u, \Lambda(u) = 1, \text{ and } \Lambda(u') = 2 \\ 
| W^{(\ell-1)}_{uu'} | \cdot \hat \alpha_{s,t,X} (u,u') & \text{if } \Lambda(u) = \ell-1 \text{ and } \Lambda(u') = \ell \text{ for } 2 < \ell < M\\
| W^{(M-1)}_{uu'} | \cdot \nu_{s,t,X}(u) & \text{if } t = u', \Lambda(u) = M-1, \text{ and } \Lambda(u') = M
\end{cases}\]
where
\[\hat \alpha_{s,t,X} (u,u') := \frac{1}{|C_t|} \sum_{x \in C_t} 1_{\text{supp} (a_{(u)})}(x) \cdot 1_{\text{supp} (a_{(u')})}(x) \cdot | a_s(x) |.\]

Figures \ref{fig:mnist1} and \ref{fig:mnist2} show $\chi \equiv \chi_{\bullet,d(s,t)}^{(s,t)}$ for st-DAGs formed between input and output neurons as described above with $k = 64$ for two MLPs trained on MNIST. Digits are discernible, indicating the utility of $\chi$ for dynamically analyzing MLPs in a way that faithfully captures real structure and that may support mechanistic interpretability \citep{somvanshi2026bridging}. Again, we stress that the number of vertices per st-DAG layer and the number of arcs between st-DAG layers are held constant by the construction above: the structure that is captured in the figures precisely encodes nonconvex activation geometry between input and output neurons and for each true class.

% \begin{figure}[htbp]
%     \centering
%     \includegraphics[width=.48\textwidth, trim={0 0 5 0mm}, clip]{mhec_mnist_deep-20260602_1702.png}
%     \includegraphics[width=.48\textwidth, trim={0 0 5 0mm}, clip]{mhec_mnist_deepest-20260602_1652.png}
%         \caption{Left: $\chi^{(s,t)}_{\bullet,d(s,t)}$ for st-DAGs formed between input and output neurons as described in the main text. The underlying ReLU MLP has three hidden layers of width 128. Rows correspond data in true classes and columns correspond to data in predicted classes; blank entries correspond to zeros in the confusion matrix. Subpanels show the Euler characteristic per input neuron/pixel; all are on a common linear grayscale shown to the right and trimmed at extremes. Right: as for the left panel, but for a MLP with five hidden layers of width 128.}
%     \label{fig:mnist}
% \end{figure}

\begin{figure}[htbp]
    \centering
    \includegraphics[width=.7\textwidth, trim={0 0 65 0mm}, clip]{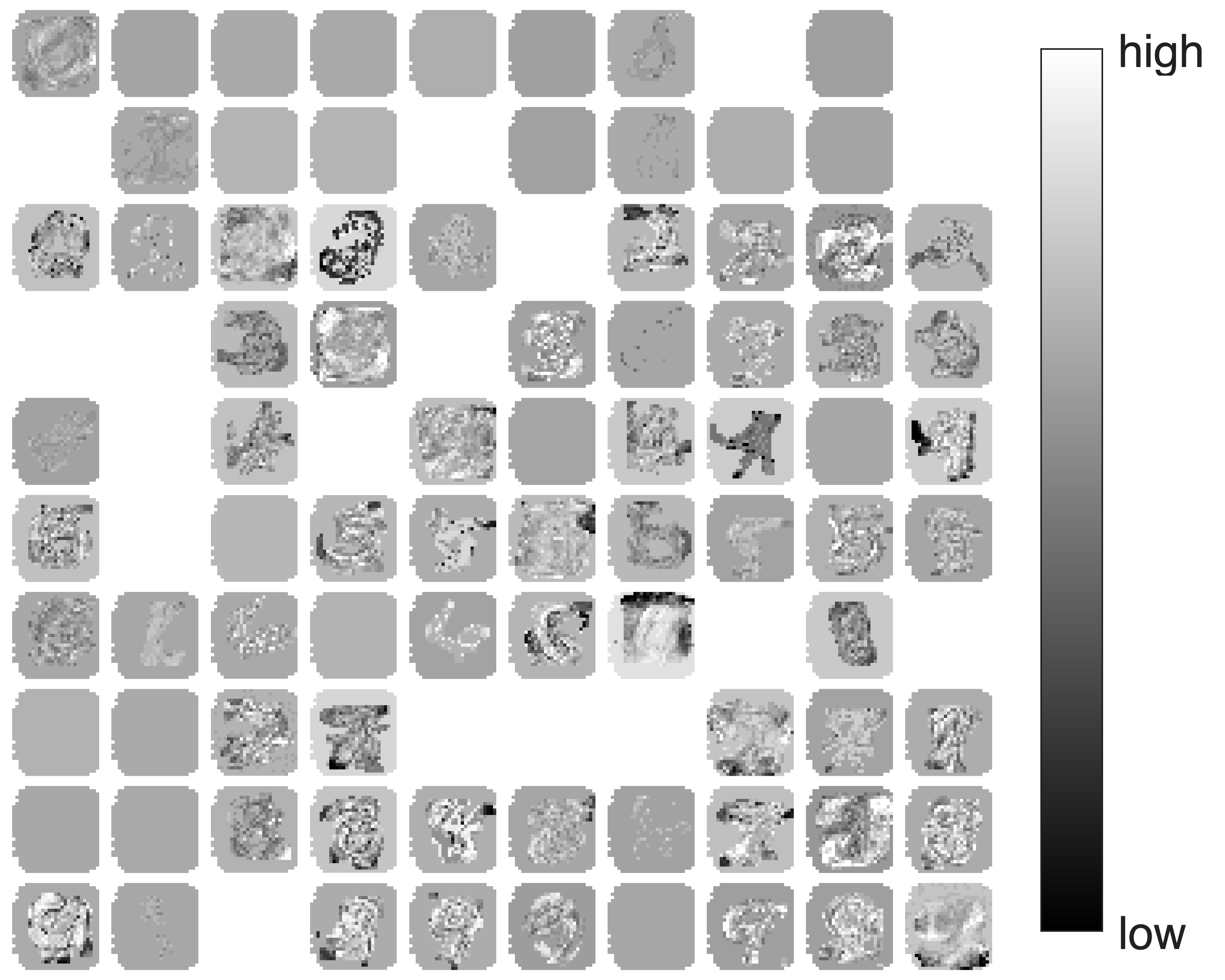}
        \caption{$\chi^{(s,t)}_{\bullet,d(s,t)}$ for st-DAGs formed between input and output neurons as described in the main text. The underlying ReLU MLP has three hidden layers of width 128. Rows correspond data in true classes and columns correspond to data in predicted classes; blank entries correspond to zeros in the confusion matrix. Subpanels show the Euler characteristic per input neuron/pixel; all are on a common linear grayscale trimmed at extremes, with black indicating the lower end and white indicating the upper end.}
    \label{fig:mnist1}
\end{figure}

\begin{figure}[htbp]
    \centering
    \includegraphics[width=.7\textwidth, trim={0 0 65 0mm}, clip]{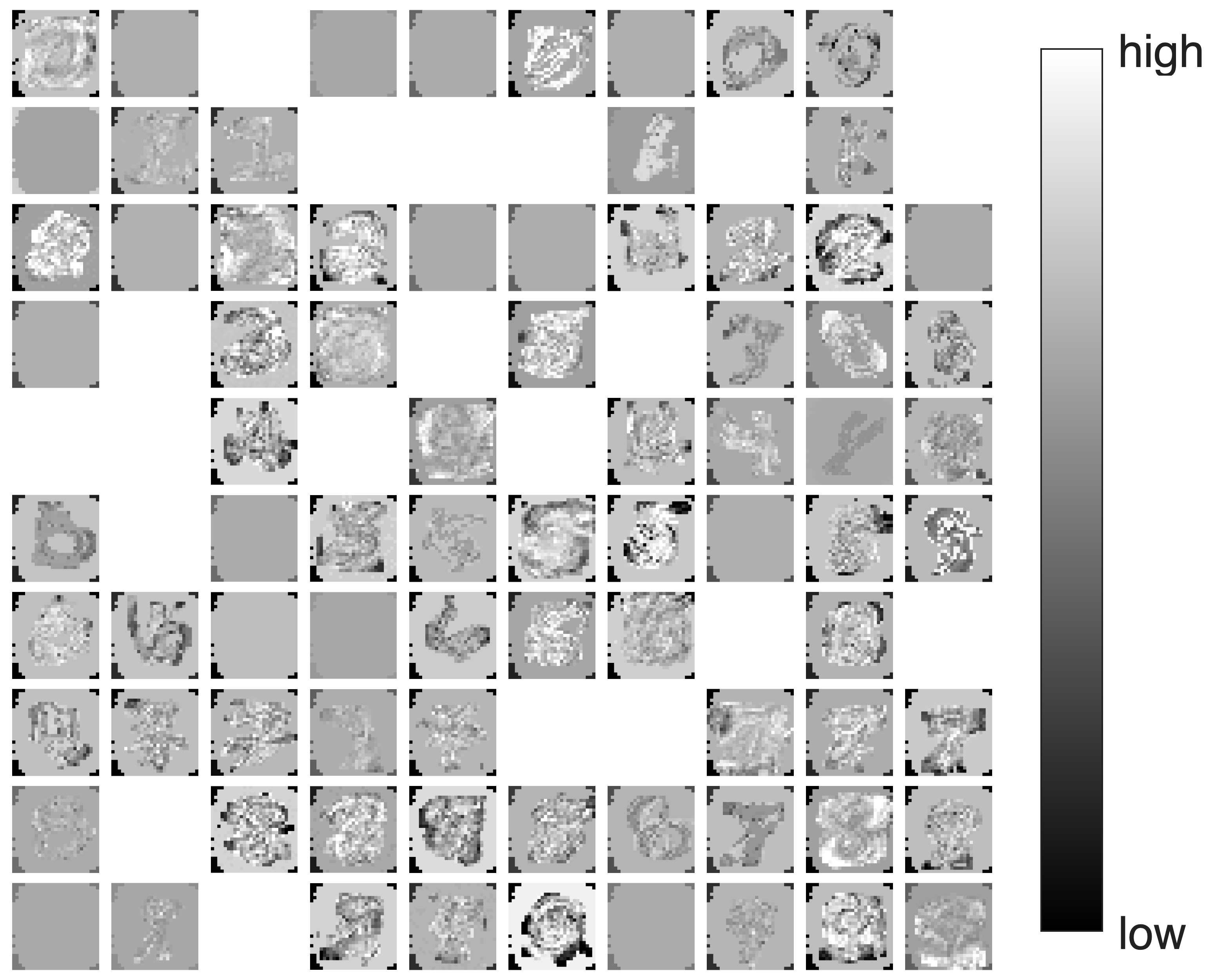}
        \caption{As in Figure \ref{fig:mnist1}, but for a MLP with five hidden layers of width 128.}
    \label{fig:mnist2}
\end{figure}

\section{Closing remark}\label{sec:remarks}

A plausible application not treated here is an analogue of the ordinary \emph{Euler characteristic transform} \citep{curry2012euler,curry2022many,munch2025invitation,rieck2025topology} by using directions in Euclidean space to ``Morse-theoretically'' define DAGs on point clouds. Any fixed notion of locality (e.g., a threshold on radius, number of nearest neighbors, or the \emph{peel neighborhood} construction of \citet{huntsman2026peel}) defines an undirected graph on a point cloud that can be readily directed to produce st-DAGs for points in general position. A key practical challenge is how to assemble Euler characteristics of different length scales across directions: one possibility is to stipulate a normalized quantized length along the Euclidean space directions.

% <TODO>
\section*{Acknowledgements}

Thanks to Evan Gorman, Giulia Menara, and Michael Robinson for illuminating conversations. I used Claude Opus to help find references, develop results, and to write the code in \S \ref{sec:floatingPointArithmetic}, but the responsibility for results, code, and for the writing itself are mine.

This research was partially developed with funding from the Defense Advanced Research Projects Agency (DARPA). The views, opinions and/or findings expressed are those of the author and should not be interpreted as representing the official views or policies of the Department of Defense or the U.S. Government. Distribution Statement “A” (Approved for Public Release, Distribution Unlimited).

\bibliographystyle{unsrtnat}
\bibliography{MHEC}

@article{abate2006unified,
  title={A unified framework for numerically inverting {Laplace} transforms},
  author={Abate, Joseph and Whitt, Ward},
  journal={INFORMS Journal on Computing},
  volume={18},
  number={4},
  pages={408--421},
  year={2006},
  publisher={INFORMS}
}

@article{adamer2024magnitude,
  title={The magnitude vector of images},
  author={Adamer, Michael F and De Brouwer, Edward and O’Bray, Leslie and Rieck, Bastian},
  journal={Journal of Applied and Computational Topology},
  volume={8},
  number={3},
  pages={447--473},
  year={2024},
  publisher={Springer}
}

@article{asao2020geometric,
  title={Geometric approach to graph magnitude homology},
  author={Asao, Yasuhiko and Izumihara, Kengo},
  journal={Homology, Homotopy and Applications},
  volume={23},
  number={1},
  pages={297--310},
  year={2020},
  publisher={International Press of Boston}
}

@article{asao2022magnitude,
  title={Magnitude homology and path homology},
  author={Asao, Yasuhiko},
  journal={Bulletin of the London Mathematical Society},
  volume={55},
  number={1},
  pages={375--398},
  year={2023},
  publisher={Wiley Online Library}
}

@article{babai1986lovasz,
  title={On {Lov{\'a}sz’} lattice reduction and the nearest lattice point problem},
  author={Babai, L{\'a}szl{\'o}},
  journal={Combinatorica},
  volume={6},
  number={1},
  pages={1--13},
  year={1986},
  publisher={Springer}
}

@book{ballester2026topological,
  title={Topological Data Analysis for Neural Networks},
  author={Ballester, Rub{\'e}n and Casacuberta, Carles and Escalera, Sergio},
  year={2026},
  publisher={Springer}
}

@article{bi2024persistent,
  title={Persistent magnitude for the quantitative analysis of the structure and stability of carboranes},
  author={Bi, Wanying and Fu, Xin and Li, Jingyan and Wu, Jie},
  journal={Journal of Computational Biophysics and Chemistry},
  volume={23},
  number={06},
  pages={741--751},
  year={2024},
  publisher={World Scientific}
}

@article{bi2025topological,
  title={Topological magnitude for protein flexibility analysis},
  author={Bi, Wanying and Feng, Hongsong and Wu, Jie and Li, Jingyan and Wei, Guo-Wei},
  journal={Royal Society Open Science},
  volume={12},
  number={12},
  year={2025},
  publisher={The Royal Society}
}

@article{bunch2021weighting,
  title={Weighting vectors for machine learning: numerical harmonic analysis applied to boundary detection},
  author={Bunch, Eric and Kline, Jeffery and Dickinson, Daniel and Bhat, Suhaas and Fung, Glenn},
  journal={arXiv preprint arXiv:2106.00827},
  year={2021}
}

@article{chang2013effects,
  title={Effects of the {LLL} reduction on the success probability of the {Babai} point and on the complexity of sphere decoding},
  author={Chang, Xiao-Wen and Wen, Jinming and Xie, Xiaohu},
  journal={IEEE Transactions on Information Theory},
  volume={59},
  number={8},
  pages={4915--4926},
  year={2013},
  publisher={IEEE}
}

@article{cho2019quantales,
  title={Quantales, persistence, and magnitude homology},
  author={Cho, Simon},
  journal={arXiv preprint arXiv:1910.02905},
  year={2019}
}

@inproceedings{chowdhury2019path,
  title={Path homologies of deep feedforward networks},
  author={Chowdhury, Samir and Gebhart, Thomas and Huntsman, Steve and Yutin, Matvey},
  booktitle={IEEE International Conference On Machine Learning And Applications (ICMLA)},
  year={2019}
}

@article{curry2012euler,
  title={{Euler} calculus with applications to signals and sensing},
  author={Curry, Justin and Ghrist, Robert and Robinson, Michael},
  journal={Advances in Applied and Computational Topology},
  volume={70},
  pages={75--145},
  year={2012},
  publisher={American Mathematical Society}
}

@article{curry2022many,
  title={How many directions determine a shape and other sufficiency results for two topological transforms},
  author={Curry, Justin and Mukherjee, Sayan and Turner, Katharine},
  journal={Transactions of the American Mathematical Society, Series B},
  volume={9},
  number={32},
  pages={1006--1043},
  year={2022}
}

@article{dingfelder2015improved,
  title={An improved {Talbot} method for numerical {Laplace} transform inversion},
  author={Dingfelder, Benedict and Weideman, JAC},
  journal={Numerical Algorithms},
  volume={68},
  number={1},
  pages={167--183},
  year={2015},
  publisher={Springer}
}

@article{emmerich2026magnitude,
  title={The magnitude of dominated sets: a {Pareto} compliant indicator grounded in metric geometry},
  author={Emmerich, Michael},
  journal={arXiv preprint arXiv:2604.18147},
  year={2026}
}

@article{gautschi2011optimally,
  title={Optimally scaled and optimally conditioned {Vandermonde} and {Vandermonde}-like matrices},
  author={Gautschi, Walter},
  journal={BIT Numerical Mathematics},
  volume={51},
  number={1},
  pages={103--125},
  year={2011},
  publisher={Springer}
}

@book{golub2013matrix,
  title={Matrix Computations},
  author={Golub, Gene H and Van Loan, Charles F},
  year={2013},
  publisher={Johns Hopkins}
}

@article{govc2021persistent,
  title={Persistent magnitude},
  author={Govc, Dejan and Hepworth, Richard},
  journal={Journal of Pure and Applied Algebra},
  volume={225},
  number={3},
  pages={106517},
  year={2021},
  publisher={Elsevier}
}

@book{hatcher2002algebraic,
  title={Algebraic Topology},
  author={Hatcher, Allen},
  year={2002},
  publisher={Cambridge}
}

@article{hepworth2017categorifying,
  title={Categorifying the magnitude of a graph},
  author={Hepworth, Richard and Willerton, Simon},
  journal={Homology, Homotopy and Applications},
  volume={19},
  number={2},
  pages={31--60},
  year={2017},
  publisher={International Press of Boston}
}

@article{hepworth2018magnitude,
  title={Magnitude cohomology},
  author={Hepworth, Richard},
  journal={Mathematische Zeitschrift},
  volume={301},
  number={4},
  pages={3617--3640},
  year={2022},
  publisher={Springer}
}

@article{huntsman2026peel,
  title={Peel neighborhoods},
  author={Huntsman, Steve},
  journal={arXiv preprint arXiv:2603.26645},
  year={2026}
}

@article{kachura2025blurred,
  title={Blurred magnitude homology of functional connectome for {ASD} diagnosis},
  author={Kachura, Alexander and Chernyshev, Vsevolod and Kachan, Oleg and Levchenko, Egor},
  journal={Frontiers in Psychiatry},
  volume={16},
  pages={1677282},
  year={2025},
  publisher={Frontiers Media SA}
}

@article{kaneta2021magnitude,
  title={Magnitude homology of metric spaces and order complexes},
  author={Kaneta, Ryuki and Yoshinaga, Masahiko},
  journal={Bulletin of the London Mathematical Society},
  volume={53},
  number={3},
  pages={893--905},
  year={2021},
  publisher={Wiley Online Library}
}

@article{leinster2013magnitude,
  title={The magnitude of metric spaces},
  author={Leinster, Tom},
  journal={Documenta Mathematica},
  volume={18},
  pages={857--905},
  year={2013}
}

@incollection{leinster2017magnitude,
  title={The magnitude of a metric space: from category theory to geometric measure theory},
  author={Leinster, Tom and Meckes, Mark W},
  booktitle={Measure Theory in Non-Smooth Spaces},
  editor={Gigli, Nicola},
  pages={156--193},
  year={2017},
  publisher={De Gruyter}
}

@book{leinster2021entropy,
  title={Entropy and Diversity: the Axiomatic Approach},
  author={Leinster, Tom},
  year={2021},
  publisher={Cambridge}
}

@article{leinster2021magnitude,
  title={Magnitude homology of enriched categories and metric spaces},
  author={Leinster, Tom and Shulman, Michael},
  journal={Algebraic \& Geometric Topology},
  volume={21},
  number={5},
  pages={2175--2221},
  year={2021},
  publisher={Mathematical Sciences Publishers}
}

@incollection{limbeck2024metric,
  title={Metric space magnitude for evaluating the diversity of latent representations},
  author={Limbeck, Katharina and Andreeva, Rayna and Sarkar, Rik and Rieck, Bastian},
  booktitle={Advances in Neural Information Processing Systems},
  year={2024}
}

@incollection{limbeck2026geometry,
  title={Geometry-aware edge pooling for graph neural networks},
  author={Limbeck, Katharina and Mezrag, Lydia and Wolf, Guy and Rieck, Bastian},
  booktitle={Advances in Neural Information Processing Systems},
  year={2026}
}

@article{munch2025invitation,
  title={An invitation to the {Euler} characteristic transform},
  author={Munch, Elizabeth},
  journal={The American Mathematical Monthly},
  volume={132},
  number={1},
  pages={15--25},
  year={2025},
  publisher={Taylor \& Francis}
}

@article{otter2018magnitude,
  title={Magnitude meets persistence: homology theories for filtered simplicial sets},
  author={Otter, Nina},
  journal={Homology, Homotopy and Applications},
  volume={24},
  number={2},
  pages={365--387},
  year={2022}
}

@article{pegolotti2022fast,
  title={Fast {M\"obius} and zeta transforms},
  author={Pegolotti, Tommaso and Seifert, Bastian and P{\"u}schel, Markus},
  journal={arXiv preprint arXiv:2211.13706},
  year={2022}
}

@inproceedings{pereverdieva2025comparative,
  title={Comparative analysis of indicators for multi-objective diversity optimization},
  author={Pereverdieva, Ksenia and Deutz, Andr{\'e} and Ezendam, Tessa and B{\"a}ck, Thomas and Hofmeyer, H{\`e}rm and Emmerich, Michael},
  booktitle={International Conference on Evolutionary Multi-Criterion Optimization},
  year={2025}
}

@book{plonka2018numerical,
  title={Numerical Fourier Analysis},
  author={Plonka, Gerlind and Potts, Daniel and Steidl, Gabriele and Tasche, Manfred},
  year={2018},
  publisher={Springer}
}

@article{rieck2025topology,
  title={Topology meets machine mearning: an introduction using the {Euler} characteristic transform},
  author={Rieck, Bastian},
  journal={Notices of the American Mathematical Society},
  volume={72},
  pages={719--727},
  year={2025}
}

@book{robinson2026computational,
  title={Computational Homological Algebra},
  author={Robinson, Michael},
  year={2026},
  publisher={Springer Nature}
}

@article{roff2025iterated,
  title={Iterated magnitude homology},
  author={Roff, Emily},
  journal={Advances in Mathematics},
  volume={468},
  pages={110210},
  year={2025},
  publisher={Elsevier}
}

@article{somvanshi2026bridging,
  title={Bridging the black box: a survey on mechanistic interpretability in {AI}},
  author={Somvanshi, Shriyank and Islam, Md Monzurul and Rafe, Amir and Tusti, Anannya Ghosh and Chakraborty, Arka and Baitullah, Anika and Chowdhury, Tausif Islam and Alnawmasi, Nawaf and Dutta, Anandi and Das, Subasish},
  journal={ACM Computing Surveys},
  volume={58},
  number={8},
  pages={1--35},
  year={2026},
  publisher={ACM New York, NY}
}

@book{stanley1997enumerative,
  title={Enumerative Combinatorics},
  volume={1},
  author={Stanley, Richard P},
  journal={Cambridge},
  year={1997}
}

@incollection{wachs2007poset,
  author = {Wachs, Michelle},
  title = {Poset topology: tools and applications},
  editor = {Miller, Ezra and Reiner, Victor and Sturmfels, Bernd},
  booktitle = {Geometric Combinatorics},
  year={2007},
  publisher={AMS}
}

\appendix

\section{\label{sec:code}Fast Euler characteristic code}

\subsection{\label{sec:floatingPointArithmetic}Floating-point code}

\begin{footnotesize}
    \begin{verbatim}
        function chi = magnitude_homology_euler_st(adj)

        % Magnitude homology Euler characteristic of a st-DAG with adjacency matrix
        % adj, unique source 1, and unique target size(adj,1).
        %
        % Uses a reasonable Ansatz to mitigate the intrinsically horrible
        % conditioning of a generalized Vandermonde matrix, and Babai's nearest
        % plane approximation for solving integer least squares.
        %
        % TODO: Generalize to blurred magnitude homology, for which see p.15 of
        % https://arxiv.org/pdf/1911.11016
        
        %%
        dag = digraph(adj);
        assert(isequal(1,find(indegree(dag)==0)),"not unique source at 1");
        assert(isequal(size(dag.Nodes,1),find(outdegree(dag)==0)),...
            "not unique target at last index");
        
        %% Magnitude homology Euler characteristic step 0: lengths of st-paths
        L_st = [];
        adj_pow = eye(size(adj));
        ell = 0;
        while any(adj_pow,"all")
            adj_pow = adj_pow*adj;
            ell = ell+1;
            if adj_pow(1,size(adj_pow,2))
                L_st = [L_st,ell]; %#ok<*AGROW>
            end
        end
        
        %% Magnitude homology Euler characteristic step 1: choose q
        % Vandermonde(-like) matrices are horribly conditioned: see, e.g.
        % https://doi.org/10.1007/s10543-010-0293-1. Accordingly, an Ansatz to make
        % the matrix A below as well-conditioned as possible is to take q
        % approximately constant, but not with _too good_ of an approximation.
        q = (1+1/numel(L_st)).^(1:numel(L_st));
        q = q(:);
        assert(numel(q)==numel(L_st),"numel(q)==numel(L_st)");
        
        %% Magnitude homology Euler characteristic step 2: build matrices A
        A = nan(numel(q),numel(L_st));
        for j = 1:numel(L_st)
            A(:,j) = q.^-L_st(j);
        end
        
        %% Magnitude homology Euler characteristic step 3: build vectors b
        % Entries of b are upper right elements of matrix inverses
        d_dag = distances(dag);
        b = nan(numel(q),1);
        e1 = [1;zeros(size(adj,1)-1,1)];
        for ii = 1:numel(q)
            b(ii) = e1.'*((q(ii).^-d_dag)\flipud(e1));  % st (upper right) entry
        end
        
        %% Magnitude homology Euler characteristic step 4: solve A*chi = b
        % Use an approximate solver for integer least squares: details below. Bad
        % conditioning means that A\b will generally not be integral in numerics,
        % though it is in principle and sometimes in practice.
        chi = A\b;
        if any(abs(chi-round(chi))>1e-4)
            % Not integer, use Babai to recompute
            chi = babai(A,b);
        end   
        % The absolute infinity norm residual should be much smaller than 1/2
        residual_inf = norm(A*chi-b,inf);
        if residual_inf > 0.05
            warning("norm(A*chi-b,inf) = "+string(residual_inf)+...
                ", not much smaller than 1/2");
        end
        
        end
        
        %% Local function for approximate solving integer least squares
        % Given A*x = y, find integer x approximately minimizing vecnorm(A*x-y)
        % using nearest plane approach from https://doi.org/10.1007/BF02579403. For
        % the explicit connection (and notation), see equation (10) of
        % https://doi.org/10.1109/TIT.2013.2253596
        function x = babai(A,y)
            [Q,R] = qr(A,0); 
            y_tilde = Q'*y;
            x = zeros(size(R,2),1);
            for i = size(R,2):-1:1
                x(i) = round((y_tilde(i)-R(i,i+1:end)*x(i+1:end))/R(i,i));
            end
        end
    \end{verbatim}
\end{footnotesize}

\subsection{\label{sec:polynomialArithmetic}Polynomial arithmetic code}

        % https://claude.ai/share/98f55def-f72a-4bfc-b0a3-46bc36ff8864
        %

\begin{footnotesize}
    \begin{verbatim}
        function chi = magnitude_homology_euler_st_poly(adj)
        
        % Magnitude homology Euler characteristic of a st-DAG with adjacency matrix
        % adj, unique source 1, and unique target size(adj,1). Computes the
        % (s,t)-component Euler characteristic at each st-path length by exact
        % polynomial back-substitution on the q-zeta matrix Z_q with entries
        % (Z_q)_{uv} = q^{d(u,v)} if d(u,v) < Inf and 0 otherwise. The result
        % chi(j) is the coefficient of q^{L_st(j)} in ((Z_q)^{-1})_{1,end}.
        %
        % This is the exact-arithmetic counterpart of magnitude_homology_euler_st.
        % It requires no toolboxes; only base MATLAB graph routines (digraph,
        % distances, toposort). It avoids the generalized Vandermonde conditioning,
        % the q-Ansatz, and the Babai heuristic of the floating-point routine.
        
        %%
        dag = digraph(adj);
        assert(isequal(1,find(indegree(dag)==0)),"not unique source at 1");
        assert(isequal(size(dag.Nodes,1),find(outdegree(dag)==0)),...
            "not unique target at last index");
        
        %% Step 0: lengths of st-paths
        L_st = [];
        adj_pow = eye(size(adj));
        ell = 0;
        while any(adj_pow,"all")
            adj_pow = adj_pow*adj;
            ell = ell+1;
            if adj_pow(1,size(adj_pow,2))
                L_st = [L_st,ell]; %#ok<*AGROW>
            end
        end
        
        %% Step 1: distance matrix (Inf for unreachable pairs)
        d_dag = distances(dag);
        
        %% Step 2: back-substitution for f(v) := ((Z_q)^{-1})(v, t) as polynomial
        % Polynomial coefficient vectors in ascending order:
        %   f{v} = [c_0, c_1, c_2, ...]   means    c_0 + c_1*q + c_2*q^2 + ...
        %
        % From Z_q * (Z_q)^{-1} = I one has the back-substitution
        %   f(t) = 1
        %   f(v) = -sum_{w != v, d(v,w) < Inf} q^{d(v,w)} * f(w)         for v != t.
        %
        % Process v in reverse topological order so that when computing f(v),
        % every f(w) reachable from v has already been computed. All polynomial
        % coefficients are integers throughout; multiplication by q^{d(v,w)} is
        % just a prepending of d(v,w) zeros, and additions extend by padding.
        
        ord = toposort(dag);             % ord(1) = 1 (source), ord(end) = n (target)
        n = size(adj,1);
        f = cell(n,1);
        f{n} = 1;
        for k = (n-1):-1:1
            v = ord(k);
            poly_v = 0;
            for w = 1:n
                if w == v, continue, end
                dvw = d_dag(v,w);
                if isinf(dvw), continue, end
                poly_v = polyadd(poly_v, -[zeros(1,dvw), f{w}(:)']);
            end
            f{v} = poly_v;
        end
        
        %% Step 3: extract coefficients of f(source) at lengths in L_st
        P = f{1}(:);
        chi = zeros(numel(L_st),1);
        for j = 1:numel(L_st)
            L = L_st(j);
            if L+1 <= numel(P)
                chi(j) = P(L+1);
            end
        end
        
        % All arithmetic above is exact integer arithmetic in double precision, so
        % chi must be exactly integer. A failure here indicates that intermediate
        % coefficients exceeded the 2^53 mantissa of double.
        assert(all(chi == round(chi)),"chi not integer (numerical loss of precision)");
        
        end
        
        %% Local function: polynomial addition in ascending-coefficient form
        function p = polyadd(a,b)
            a = a(:)'; b = b(:)';
            na = numel(a); nb = numel(b);
            if na < nb, a(nb) = 0; end
            if nb < na, b(na) = 0; end
            p = a + b;
        end        
    \end{verbatim}
\end{footnotesize}

\end{document}